\documentclass[reqno]{amsart}
\usepackage{amssymb}
\usepackage[matrix,arrow,tips,curve,ps]{xy}
\input xypic
\usepackage{comment}
\SelectTips{cm}{10}
\usepackage{enumerate}
\usepackage{color}
\newtheorem{thm}{Theorem}[section]
\newtheorem{lem}{Lemma}[section]

\newtheorem{prop}{Proposition}[section]

\theoremstyle{remark}
\newtheorem{remark}{Remark}[section]
\newtheorem{claim}{Claim} [section]
\newtheorem{prob}{Problem} [section]

\newcommand{\FF}{{\mathbb F}}

\begin{document}
\title[On a theorem of Castelnuovo]{On a theorem of Castelnuovo 
and applications to moduli}
\author{Abel Castorena}  
\address{Instituto de Matem\'aticas(Universidad Nacional Aut\'onoma de M\'exico)\\
Unidad Morelia\\
Apdo. Postal 61-3 (Xangari). C.P. 58089\\
Morelia, Michoac\'an, M\'exico.}
\email{abel@matmor.unam.mx}

\author{Ciro Ciliberto}
 \address{Dipartimento di Matematica,
Universit\`a di Roma Tor Vergata,
Via della Ricerca Scientifica, 00133 Roma, Italy.}
\email{cilibert@mat.uniroma2.it}

\thanks{The first author has been partially supported by CONACYT(M\'exico) Grants 058486, 48668 and PAPIIT(UNAM) Grant IN100909-2. The second author is a member of GNSAGA of INDAM}
\subjclass[2000]{Primary 14C20. Secondary 14J26.}
\keywords{Castelnuovo theorem,  moduli of curves.}

\maketitle

\begin{abstract} In this paper we prove a theorem stated by Castelnuovo in \cite {CA2} which bounds the dimension
of linear systems of plane curves in terms of  two invariants, one of  which is the genus of the curves in the system.
This extends a previous result of Castelnuovo--Enriques (see \cite {CR}).  We classify linear systems
whose dimension belongs to certain intervals which naturally arise from Castelnuovo's theorem. Then 
we make an application to the following moduli problem:
what is the maximum number of moduli of curves of geometric genus $g$ varying in a linear system on a surface? It turns out that, for $g\ge 22$, the answer
is $2g+1$, and it is attained by trigonal canonical curves varying on a balanced rational normal scroll. 
\end{abstract}
\section*{Introduction}

This paper has been originated by the following problem (see Problem \ref {prob:moduli}).  Consider the set $\mathcal X_g$, with $g\geq 2$,  of all linear systems $\mathcal L$ of curves on a surface $X$ such that the general curve
of $\mathcal L$ is irreducible, with geometric genus $g$. 
For such an $\mathcal L$, consider its image in $\mathcal M_g$  via the obvious (rational) moduli map.  
What is the maximum dimension of this image (called the \emph{number of moduli} of $\mathcal L$) when $\mathcal L$ varies in $\mathcal X_g$? 

A naive expectation is that, the larger the dimension of  $\mathcal L$, the larger its number of moduli. So a related question is:
what is the maximum of $r=\dim(\mathcal L)$ as $\mathcal L$ varies in $\mathcal X_g$? This has a classical answer which goes back to Castelnuovo \cite {CA1} and Enriques \cite {En}. They proved  an important result (see Theorem \ref
{thm:CastEnr}) to the effect that $r\le 3g+5$, with three exceptions wich, up to birational equivalence, are the following: either $\mathcal L$ is the linear system of plane cubics or the rational map determined by $\mathcal L$ realizes $X$ as a scroll. In the former case the number of moduli is 0, in the latter it is 1. Castelnuovo and Enriques also classified the cases in which
the bound $r=3g+5$ is attained: $X$ is then rational and $\mathcal L$ is either (up to birational equivalence) the linear system of plane curves of degrees $2$ or $4$ or a suitable system of hyperelliptic curves. Castelnuovo--Enriques' theorem has been rediscovered and/or reconsidered a few times in the course of the years: see \cite{CR} for classical and
more recent references. 

At about the same time, Castelnuovo stated in  \cite{CA2}, with a rather sketchy proof, a more general and interesting theorem which classifies linear systems on rational surfaces with $r>g$ (see Theorem \ref {thm:Castthm}). Castelnuovo's argument is
 based on an ingenuous application of adjunction and on a basic inequality (see Theorem \ref {thm:Castineq}) which  improves the original Castelnuovo--Enriques theorem. Castelnuovo's Theorem 
 \ref {thm:Castthm} is a very interesting result in birational geometry of surfaces, and more recent developments, e.g. \cite[Corollary (1.1)]  {R}, are reminiscent of it.  Section \ref {sec:Castthm} is devoted to prove, following and clarifying Castelnuovo's original idea,  Castelnuovo's inequality and  theorem. 
 
Castelnuovo's theorem applies to our original moduli problem, which we take up in \S \ref {sec:moduli}, where we answer our original problem, at least when $g$ is large enough.
 We prove (see Theorem \ref {thm:moduli}) that the maximum number of moduli
 of a linear system of curves of genus $g\ge 22$ is $2g+1$ and it is attained by the linear systems of trigonal canonical curves on a balanced rational normal scroll in 
$\Bbb P^ {g-1}$ (the bound $g\ge 22$ could be improved, but we thought it useless to dwell on this here). 
It is remarkable that this maximum  is not achieved by linear systems of the largest dimension $3g+5$ compatible with a non--trivial map to moduli: indeed, as we said, they consist of hyperelliptic curves, and in fact they dominate the hyperelliptic locus, which has dimension $2g-1$ (see Theorem  \ref {prop:modcast}).
The proof of Theorem \ref {thm:moduli} relies on Castelnuovo's theorem, on the concept of
\emph{Castelnuovo pairs},  on their classification and 
related computation of  moduli (see \S \ref {ssec:maxdim}). 

In conclusion, it is worth mentioning, on the same lines as the problem 
considered here, another more fascinating and complicated one (attributed to F. O. Schreyer): what is,  for large enough $g$,  the maximum dimension of a rational [or, respectively unirational, uniruled, rationally connected] subvariety of $\mathcal M_g$?

 \section*{Notation, conventions and generalities}\label{sec:gen}      
We use standard notation in algebraic geometry. In particular, the simbol $\equiv$ denotes linear equivalence of divisors.  If $D$ is a divisor on a smooth, projective variety
$X$, $|D|$ is the complete linear system of $D$. 
If $\mathcal L$ is a linear system of divisors on $X$ of dimension $r$,  $\phi_{\mathcal L}:X\dasharrow \Bbb P^r$ is the  rational map defined by $\mathcal L$. The system $\mathcal L$ is said to be \emph{simple} if $\phi_{\mathcal L}$ maps $X$ birationally to its image. 

 
Let $X$ be a smooth irreducible projective surface. As usual we denote by $K:=K_X$ a \emph{canonical divisor}, $q:=q(X):=h^1(X,\mathcal O_X)$ the {\it{irregularity}}, $p_g:=p_g(X):=h^0(X,\mathcal O_X(K))$  the {\it{geometric genus}} of $X$.

Let $D$ be a divisor on $X$. 
We will say that $D$ is a ${\it{curve}}$ on $X$ if it is effective. If $D$ is a reduced curve on $X$,  
the {\it{geometric genus}} $g$ of $D$ is the arithmetic genus of the normalization of $D$. 
Often we will simply call $g$  the \emph{genus} of $D$.  We will use the notation $d=D^ 2$ and $r=\dim(\vert D\vert)$.
Moreover $D'\equiv K+D$ is an \emph{adjoint divisor}  and $\vert D'\vert$ the \emph{adjoint linear system} to $D$.
The system $\vert D\vert$ is called \emph{non--special} if it is either  empty or $h^ 1(X,\mathcal O_X(D))=0$.

Suppose there is a morphism $f: X\to Y$, contracting a curve $C$ of $X$ to a smooth point $p$ of a surface $Y$ and induces an isomorphism between $X-C$ and $Y-\{p\}$. The divisor $E$, supported on $C$, which is the scheme theoretical fibre of $f$ over $p$, is called a $(-1)$--\emph{cycle}, or a $(-1)$--\emph{curve} if $E=C$ is irreducible. 

We will consider pairs $(X,D)$, with $X$ a smooth irreducible projective surface and $D$ a curve on it. We will extend attributes of $D$ (like being \emph{nef, big, ample} etc.) or of $\vert D\vert$ (like being \emph{simple, special, very ample} etc.) to the pair $(X,D)$. We say that $(X,D)$ is:

\begin{itemize}

\item   {\it{minimal}} if there is no (-1)--curve $C$ on $X$ such that $D\cdot C=0$;

\item  a  {\it{h-scroll}}, if there is a smooth rational curve $F$ on $X$ such that $F^2=0$ and $D\cdot F=h$. A 1-scroll will be simply called a {\it{scroll}}.
\end{itemize}


There are obvious notions of morphism, isomorphism, rational and birational maps between pairs (see \cite {CaCi}). We are mainly interested in birational invariants of the linear system $\vert D\vert$ on $X$. 
If  $\vert D\vert$ has no fixed curves and its general curve is irreducible, then by blowing up the base locus of $\vert D\vert$
we may assume $\vert D\vert$ is base point free and the general curve of $D$ is smooth. So we will often assume this is the case.
In addition we may assume $(X,D)$ is minimal by successively contracting all $(-1)$--curves $E$ with $D\cdot E=0$. 

If $X\cong\Bbb P^2$ and $\ell$ is a line, the pair $(X,D)$ with $D\equiv m\ell$ will be called a \emph{$m$--Veronese pair}. 


As usual, we will denote by $\Bbb F_a$ the \emph{Hirzebruch surface} $\Bbb P(\mathcal O_{\Bbb P^1}\oplus\mathcal O_{\Bbb P^1}(-a))$. The Picard group of $\Bbb F_a$ is freely generated by the classes of the divisors:
$E$, a curve with $E^ 2=-a$ (unique if $a>0$), and $F$,  a \emph{ruling}, i.e. a fibre of the structure morphism $f:\Bbb F_a\to\Bbb P^1$. 
One has $F^2=0, F\cdot E=1$.  A divisor $D\equiv \alpha E+\beta F$ is nef as soon as $D\cdot E=\beta-a\alpha \ge 0$. 
If $\alpha=1$ and $\beta\ge a$ then $\phi_{\vert D\vert}$ birationally maps $\Bbb F_a$ to a \emph{rational normal scroll} of degree $s-1$ in $\Bbb P^ s$, with $s=2\beta-a+1$. 
A pair $(X,D)$ with $X\cong \Bbb F_a$ and $D\equiv 2E+(a+g-1)F$ is nef, the general curve in $\vert D\vert$ is smooth of genus $g$ and $r=3g+5$. Such a pair is called a $(a,g)$-{\it{Castelnuovo pair}} (see \cite {CR}).

\section{Castelnuovo's theorem}\label{sec:Castthm}

\subsection{Castelnuovo--Enriques  theorem}\label {ssec:CastEnr}

We recall  the following theorem which extends 
results of  Castelnuovo \cite {CA1} and Enriques \cite {En}  (see  \cite[Theorem 7.3] {CR} and  \cite{CR} also 
for classical and recent references):

\begin{thm} [Castelnuovo-Enriques theorem]\label{thm:CastEnr} Let $(X,D)$ be a pair 
with $D$ an irreducible curve.  Assume $d> 0$ and
$(X,D)$ not a  scroll. Then:\
 \begin{equation} \label{grado} d\leq 4g+4+\epsilon
\end{equation}
\noindent where $\epsilon=1$ if $g=1$ and $\epsilon=0$ if $g\neq
1$. Consequently one has:
\begin{equation} \label{erre} r\leq 3g+5+\epsilon
\end{equation}
\noindent and the equality holds in (\ref {grado}) if and only if
it holds in (\ref{erre}).\par
 If, in addition, the pair $(X,D)$ is minimal, then the equality
holds in (\ref {erre}), if and
only if one of the following happens:\par
\begin{itemize}
\item [(i)] $g=0$, $r=5$, and $(X,D)$ is a $2$--Veronese pair;\par
\item [(ii)] $g=1$, $r=9$, and $(X,D)$ is a $3$--Veronese
pair;\par
\item [(iii)] $g=3$, $r=14$, and $(X,D)$ is a $4$--Veronese
pair;\par
\item [(iv)] $(X,D)$ is a $(2,n+g+1)$--Castelnuovo pair on $X\cong  \FF_n$,
$n\geq 0$.\end{itemize} \end{thm}

\subsection{Castelnuovo's inequality}\label{ssec:Castineq}

In this section we prove a result of Castelnuovo \cite{CA2}, which 
specifies \eqref {grado}. 

We consider here minimal pairs $(X,D)$ with $p_g=q=0$, $D$ an irreducible, smooth curve of genus $g\geq 2$,  
with  $d\ge 1$ and $r\ge 1$, hence $D$ is nef. By  \cite[Proposition 7.1]{CR},  an adjoint curve $D'\equiv K+D$ is nef and 
\begin{equation}\label {eq:a}
d\le 4(g-1)+K^ 2\le 4g+5
\end{equation}
which is basically the proof of \eqref {grado} (in the last inequality we used Miyaoka--Yau inequality). Moreover
$\dim(\vert D'\vert)=g-1$.  Set $\vert D'\vert =P+\vert M\vert$, where $P$ is the fixed divisor and $\vert M\vert$ is the movable part, called the \emph{pure adjoint system} of $D$. 
We set $g':=p_a(M)$ and $d'=M^2$. One has $M\cdot D=2g-2$ and $P\cdot D=0$ and for all curves  $E\leq P$ one has $E^2<0$. 

\begin{lem}\label{lem:nonspec} In the above setting, if $d\ge 5$ and $\vert D\vert$ is non--special, then $P=0$.\end{lem}

\begin{proof} Reider's Theorem (see \cite {BS, Re}) implies that, if $x$ is a base point of $\vert D'\vert$, there is
an irreducible curve $A$ containing $x$, such that 
either $A\cdot D=1, A^2=0$ or $A\cdot D=0, A^2=-1$. 

Let $E$ be an irreducible curve contained in $P$. 
For all $x\in E$, we have a curve $A_x$ as above.  If $A_x\cdot D=1$ then $A_x\neq E$.
Moreover $A_x^ 2=0$, so
$A_x$ moves in a base point free pencil $\vert A\vert$. Since $A\cdot D=1$, we would have $g=0$, a contradiction. Hence $A_x\cdot D=0$ and $E=A_x$. This shows that $E^ 2=-1$. 

Since $ D\cdot E=0$ and $r\ge 1$,  then $D-E$ is effective. We have the exact sequence $0\to\mathcal O_X(D-E)\to\mathcal O_X(D)\to\mathcal O_E(D)\cong \mathcal O_E\to 0$, which yields the exact sequence  $ H^1(X,\mathcal O_X(D))\to H^1(E,\mathcal O_E)\to H^2(X,\mathcal O_X(D-E))$. Since $p_g=0$, the last space is $0$, and the first is $0$ by assumption. Hence  $h^1(E,\mathcal O_E)=0$, then $E$ is rational. 

In conclusion, $E$ is a $(-1)$--curve such that $D\cdot E=0$, contradicting the minimality assumption.\end{proof}

 If  $\vert M\vert$ is composed with a pencil $\vert L\vert$, then $\vert M\vert=\vert (g-1)L\vert$, $\dim(\vert L\vert)=1$ and $\vert L\vert$ has no base points on $D$. Then $D\cdot L=2$,  $D$ is hyperelliptic  and:

\begin{enumerate}

\item  either $\vert D\vert$ cuts out a base point free $g^1_2$ on the general curve $L$ of $\vert L\vert$, hence there is a birational involution $\iota: X\dasharrow X$ that fixes all curves in
$\vert D\vert$, which then is not simple (in this case we say that $\vert D\vert$ is \emph{composed} with the involution $\iota$);
\item  or  $\vert D\vert$ cuts out a $g^2_2$ on $L$ and $\vert L\vert$ is a pencil of curves of genus 0. 
\end{enumerate}
If $d\ge 5$, the index theorem implies $L^ 2=0$. 

 
\begin{thm}[Castelnuovo's inequality]\label{thm:Castineq} Let $(X,D)$ be minimal with $D$ smooth and irreducible, with $g\ge 2$, 
$d\ge 1$ and $r\ge 1$.
Assume that either $D$ is not hyperelliptic or $\vert D\vert$ is not composed with an involution of $X$. 
Then 
\begin{equation}
\label{eq:0}
d\leq 3g+7-g'
\end{equation}
and equality holds if and only if  $X=\mathbb P^2$.\end{thm}


\begin{proof} 
Suppose first  $\vert M\vert $ is composed with a pencil $\vert L\vert$. Then $D$ is hyperelliptic 
with $D\cdot L=2$ and $\vert D\vert$ is not composed with an involution of $X$. Thus the curves in $\vert L\vert$ have genus 0, so $X$ is rational, $L^ 2=0$ and $g'=2-g$. 
By \eqref {eq:a} we have $d\leq 4g+5=3g+7+g-2=3g+7-g'$. If equality holds then $K^2=9$ hence $X=\Bbb P^2$.  

Suppose next  $\vert M\vert$ is not composed with a pencil, hence $d'>0$. We have  
\[
h^0(M,\mathcal O_M(M))=h^0(X,\mathcal O_X(M))-1=h^0(X,\mathcal O_X(K+D))-1=g-1
\]
then $\vert \mathcal O_M(M)\vert=g^{g-2}_{d'}$. We have two cases: (a) $K\cdot M\ge 0$; (b) $K\cdot M< 0$


In case (a), since $D'$ is nef, one has
\begin{equation}
\label{eq:1} 
(K+D)^2\geq(K+D)\cdot M\geq D\cdot M\geq 2g-2
\end{equation} 
and equality implies $K\cdot M=0$. Let $R\equiv D-M\equiv P-K$. We have two subcases: (a1) $\vert R\vert=\emptyset$; (a2) $R$ is effective. 
In case  (a1), one has $1>\chi(\mathcal O_X(R))
$, which reads $d<3g-g'-2$ and \eqref{eq:0} holds.


In case (a2), one has $2g-2=(K+D)\cdot D=(K+D)\cdot(M+R)\geq(K+D)\cdot M=(K+M+R)\cdot M\geq 2g'-2$, then $g\geq g'$ and, by \eqref {eq:1},  $g+g'-2\leq 2g-2\leq(K+D)^2=4g-4+K^2-d$, therefore \eqref {eq:0} holds. If equality holds then $K^2=9$ and $K\cdot M=0$ (because equality holds in \eqref {eq:1}).  This cannot happen on a surface of general type because $d'>0$, thus $X\cong \Bbb P^2$.

In case (b) one has $d'>2g'-2$. Then  $g^{g-2}_{d'}$ is not special,  hence $g-1=h^0(M,\mathcal O_M(M))=d'-g'+1$. Since $K+D\equiv P+M$ 
is nef, then $g+g'-2=d'=M^2\leq (K+D)^2=4g-4+K^2-d$, so $d\leq 3g-g'+(K^2-2)\leq 3g+7-g'$ and if equality holds, then $K^ 2=9$ and, as above, $X\cong \Bbb P^ 2$. 


Finally, if $X=\Bbb P^2$ then  \eqref {eq:0} holds with equality. \end{proof}

\subsection{Castelnuovo's theorem}\label{ssec:Castthm}

In this section we prove a theorem of Castelnuovo stated in  \cite{CA2} which classifies linear systems on rational surfaces with $r>g$. A remark is in order. 

\begin {remark}\label{rem:g} Consider a pair $(X,D)$ with $D$ smooth, irreducible such that $r>g$.
Then $h^ 0(D, \mathcal O_D(D))>g$, hence $h^ 1(D, \mathcal O_D(D))=0$. Thus
$d-g+1=h^ 0(D, \mathcal O_D(D))>g$ and therefore $d\ge 2g$, so that $K\cdot D<0$, which implies that
$X$ has negative Kodaira dimension. This shows that the rationality assumption on $X$ in Theorem  \ref 
{thm:Castthm}  below is no restriction. In this case, one has $h^ 1(X, \mathcal O_X(D))=0$, i.e. $\vert D\vert$ is non--special. \end{remark}

\begin{thm} [Castelnuovo's theorem]\label{thm:Castthm}  Let $(X,D)$ 
be minimal with $D$ smooth, irreducible, of genus $g\ge 2$, $X$ rational,  $\vert D\vert$ is not composed with an involution of $X$ and
\begin{equation}
\label{eq:Cast}
r\geq\tau(\mu,g):=\frac{(\mu+2)g+\epsilon_{\mu}}{\mu}+2\mu+3\hskip20mm 
\end{equation}
where
\begin{equation*}
\epsilon_{\mu}=\left\{\begin{array}{rl}
1 &\text{for }\mu\text{ odd}\\ 
2 &\text{for }\mu\text{ even}\end{array}\right.
\end{equation*} 
Then:
\begin{itemize}
  
  
\item  [(i)] either there is a  birational morphism $\phi:X\to \Bbb P^2$ such that $|D|$ is the proper transform of a linear system of plane curves of degree $m\leq 2\mu+1$ with base points of multiplicity $k\leq\ulcorner \frac{\mu}{2}\urcorner-1$,


\item  [(ii)] or $(X,D)$ is a $m$--scroll with $m\leq\mu$, and precisely there is a birational map $\phi:X\dasharrow\Bbb F_a$, for some $a\ge 0$,  such that $|D|$ is the proper transform of a linear system of $m$--secant curves to the ruling of $\Bbb F_a$, with base points of multiplicity 
$k\leq\ulcorner\frac{\mu}{2}\urcorner-1$.
\end{itemize}
\end{thm}


\begin{proof}    Since $\tau(\mu,g)>g$, Remark \ref {rem:g} applies.

For $\mu=1$ one has $\tau(1,g)=3g+6$ and the assertion follows by 
Castelnuovo-Enriques Theorem \ref {thm:CastEnr}. So we may assume $\mu\ge 2$. 

If the general curve in $|M|$ is reducible, then $\vert M\vert$ is composed with a pencil $\vert L\vert$ of curves of genus 0 such that $L\cdot D=2$ (see the proof of Castelnuovo's inequality \ref {thm:Castineq}).  Then there is a birational morphism 
$\psi:X\to\Bbb F_n$ such that $\vert L\vert$  is the proper transform of the ruling of $\Bbb F_n$, and 
$\vert D\vert$ is the proper transform of a linear system of $2$--secant curves to the ruling of $\Bbb F_n$, with at most double base points. By performing elementary transformations based at these double base points, we find case (ii) with $\mu=2$. 
So from now on we may assume the general curve in $\vert M\vert$ to be irreducible. 

Let $\mu=2$.
We have $r=d-g+1\geq \tau(2,g)=2g+8$, which is equivalent to $d\geq 3g+7$. By Castelnuovo's inequality $3g+7\leq d\leq 3g+7-g'$ then $g'\leq 0$.  Since $M$ is irreducible, we have $g'\ge 0$, thus $g'=0$, equality holds in \eqref{eq:1}, hence $X=\Bbb P^ 2$ and we are in case (i). 

Next we assume $\mu\ge 3$ and we will make induction on $\mu$. By Castelnuovo's inequality we have $r=d-g+1\leq 2g-g'+8$. If  equality holds then $X\cong\Bbb P^2$ and $(X,D)$ is a $m$--Veronese pair, i.e.  $D\in|\mathcal O_{\Bbb P^2}(m)|$. We claim that  $m\leq 2\mu+1$, i.e.  we are in case (i).  Indeed $g={(m-1)(m-2)}/{2}$, $r= {m(m+3)}/2$ and \eqref {eq:Cast} reads 
$2m^2-6m(\mu+1)+4\mu^2+8\mu+2(\epsilon_{\mu}+2)\leq 0$. The polynomial
$h(x)=2x^2-6x(\mu+1)+4\mu^2+8\mu+2(\epsilon_{\mu}+2)$ has its critical value at $x_0={3(\mu+1)}/{2}<2\mu$, so that $h$ is strictly increasing in $[x_0,+\infty)$.  If 
$m\geq 2\mu+2$, we would have $0\ge h(m)> h(2\mu+2)=2\epsilon_{\mu}$, a contradiction. 

Now we analyse the case $r\leq 2g-g'+7$. Then $2g-g'+7\geq r\geq\tau(\mu,g)$
 implies $g\geq\frac{\mu g'+\epsilon_{\mu}}{\mu-2}+2\mu$. Thus 
\begin{equation}
\label{eq:2}
\text{dim }(|D'|)=\text{dim }(|M|)=g-1\geq\frac{\mu g'+\epsilon_{\mu-2}}{\mu-2}+2\mu-1=\tau(\mu-2,g').
\end{equation}
By induction, we may assume $r<\tau(\mu-1,g)$, hence $\tau(\mu-1,g)>\tau(\mu,g)$, which yields
$g>\mu(\mu-1)+\frac{1}{2}((\mu-1)\epsilon_{\mu}-\mu\epsilon_{\mu-1})$. Thus
\begin{equation}\label{eq:g}
g>\left\{\begin{array}{rl}
\mu(\mu-1)+\frac{\mu-2}{2}, &\text{for }\mu\text{ an even number}\\ 
\\
\mu(\mu-1)-\frac{\mu+1}{2}, &\text{for }\mu\text{ an odd number}\end{array}\right.
\end{equation}
 In particular, if $\mu\geq 3$, then $g\geq 4$  and $d\ge 2g-2\ge 6$. Then, by Lemma \ref {lem:nonspec}, $P=0$ and $\vert D'\vert=\vert M\vert$. 
 
In view of  \eqref{eq:2},  we would like to apply induction on $\vert M\vert$, which we can do only if $M$ verifies the hypotheses of the theorem. 

First, we dispose of the case $g'=1$, in which \eqref{eq:2} implies $\dim (\vert M\vert)\geq 9$, and equality holds only for $\mu=3$. Then, by Castelnuovo--Enriques Theorem \ref {thm:CastEnr},  $\mu=3$, $(X,M)$ is a 3--Veronese pair and $D$ is a smooth plane sextic, i.e. we are in case (i).
Hence from now on we may assume $g'\geq 2$.

\begin{claim}\label{cl:1} The system $\vert M\vert$ is not composed with an involution of $X$.
\end{claim}

\begin{proof} [Proof of Claim \ref {cl:1}]
Suppose that $M$ is composed with a involution $\iota$ of $X$, defined in the Zariski open subset $U$.  Consider the incidence variety
$\mathcal V$ which is the Zariski closure in $X\times X\times |D|$ of the set $
\{(p,q,D):p,q\in D, D\in|D|, \iota(p)=q\}\subset U\times U\times |D|$, with the projections $\pi_1:\mathcal V\to X\times X$,  $\pi_2:\mathcal V\to |D|$ to the factors. The image of $\pi_1$ is the graph $\Gamma$ of $\iota$.
Since $\vert D\vert$ is not composed with $\iota$, the general fibre of $\pi_1$  has dimension $r-2$.  
Hence $\mathcal V$ has an irreducible component $\mathcal W$ which dominates $\Gamma$ via $\pi_1$ and has dimension $r$. 

If $\pi_{2\vert \mathcal W}$ is surjective, then the general curve in $|D|$ is hyperelliptic.  Since $\mu\ge 3$ and $g\geq 4$, \eqref {eq:2} yields $r\ge 14$ hence $d\ge 17$.  By Reider's theorem (see again  \cite {Re, BS}), there is curve $A$ such that $0\le A\cdot D-2\leq A^2<\frac{A\cdot D}{2}<2$. The index theorem implies $A^ 2=0$, then $A\cdot D=2$ and there is a base point free pencil $\vert A\vert$  which cuts the $g^ 1_2$ on the general
curve of $\vert D\vert$. Since $\vert D\vert$ is not composed with an involution, the curves in $\vert A\vert$ have genus 0 (see the discussione before Theorem \ref {thm:Castineq}). 
Then $0=A\cdot(K+D)=A\cdot M$. Since the general curve in $\vert M\vert$ is irreducible and 
$\dim(\vert M\vert)=g-1\ge 3$,  this is a contradiction.

If $\dim (\pi_2(\mathcal W))<r$, let $D\in\text{Im }(\pi_2)$ be a general element. The general fibre of $\pi_{2}$ has dimension at most $1$, then it has dimension one, hence $\dim (\pi_2(\mathcal W))=r-1$.
Now repeat the same argument as above.
\end{proof}

The pair $(X,M)$ could be not minimal. If $E$ is a $(-1)$--curve such that $E\cdot M=0$, then $E\cdot D=1$. By contracting these  $(-1)$--curves we have a birational morphism $f: X\to X'$ and there are irreducible curves $D'$ and  $M'$ on $X'$ 
whose proper transform on $X$ are $D$ and $M$. The linear system $\vert D\vert$ is the proper transform of the sublinear system of $\vert D'\vert$ formed by the curves passing through the points which are blown--up in $f: X\to X'$. 

Finally we may apply induction to the pair $(X',M')$, and:

\begin{itemize}

\item [(i')] either  there is a birational morphism $\phi':X'\to\Bbb P^2$ and $|M'|$ is the proper transform of a linear system $|C|$ of curves of degree $d\leq 2\mu-3$ with base points of multiplicity $k\leq\ulcorner\frac{\mu}{2}\urcorner-2$;

\item [(ii')]  or there is a birational  map $\phi':X'\dasharrow \Bbb F_a$, and $|M'|$ is the proper transform of a linear system  $|C|$ of $m$--secant curves to the ruling of $\Bbb F_a$ with $m\leq\mu-2$ with base points of multiplicity $k\leq\ulcorner\frac{\mu}{2}\urcorner-2$.
\end{itemize}

In case (i'), consider $\phi=\phi'\circ f: X\to \Bbb P^ 2$. 
If $\ell$ is a line in $\Bbb P^2$, set $H=\phi^*(\ell)$.
We have 
$M\equiv dH-\sum\limits_i k_iE_i$,  where $E_i$ are $(-1)$--cycles 
contracted by $\phi$ and $k_i\leq\ulcorner \frac{\mu-2}{2}\urcorner-1$.
We have also $K_X\equiv -3H+\sum\limits_i E_i$. Then 
$D\equiv (d+3)-\sum\limits_i(k_i+1)E_i$, and we are in case (i).
The analysis of (i'') is similar and leads to  case (ii).  \end{proof}

\section{Castelnuovo pairs and their moduli}\label{sec:moduli}

Castelnuovo--Enriques Theorem \ref {thm:CastEnr} classifies minimal pairs $(X,D)$
for which \eqref {eq:2} holds with $\mu=1$.  For higher $\mu$'s we define the concept of $\mu$--\emph{Castelnuovo pairs}.

\subsection{Castelnuovo pairs}\label{ssec:maxdim}

We will call a pair $(X,D)$ as in  Castelnuovo's Theorem \ref {thm:Castthm} a
$\mu$--\emph{Castelnuovo's pair}, with $\mu\ge 2$, if 
\[
\tau(\mu-1,g)>r\ge \tau(\mu,g)
\]
which implies that \eqref {eq:g} holds (see the proof of Theorem \ref {thm:Castthm}).
If for such a pair case (i) [resp. case (ii)] occurs, we say that
it presents the \emph{planar case} [resp. the \emph{scroll case}]. 
Here we list $\mu$--Castelnuovo's pairs for $2\le \mu\le 4$. The reader may check the details.

\begin{prop}\label{prop:list} If $(X,D)$ is a $\mu$--Castelnuovo's pair with $2\le \mu\le 4$
with $D$ smooth of genus $g\ge 2$, then:
\begin{itemize}
\item [(i)] $(X,D)$ is either an $m$--Veronese pair with
 \begin{equation*}
 \begin{array}{rl}
4\le m\le 5\quad  \text{if}\quad  \mu=2\quad  \text{and}\quad r=\ulcorner \tau(2,g)\urcorner \\
6\le m\le 7 \quad  \text{if}\quad  \mu=3\quad  \text{and}\quad r=\ulcorner \tau(3,g)\urcorner+\eta  \\
8\le m\le 9 \quad  \text{if}\quad \mu=4 \quad  \text{and}\quad r=\ulcorner \tau(4,g)\urcorner+\eta
\end{array}
\end{equation*}
where, in the last two cases, $\eta=0$ if $m$ is odd and $\eta=1$ if $m$ is even, or
$X\cong \Bbb F_1$, and $D=(m-1)E+mF$ with
 \begin{equation*}
\begin{array}{rl}
m=6 \quad  \text{if}\quad  \mu=3 \\
m=8 \quad  \text{if}\quad \mu=4
\end{array}
\end{equation*}
and $r=\ulcorner \tau(\mu,g)\urcorner$ in both cases.

 \item[(ii)] $X\cong \Bbb F_a$ and $D\equiv \mu E +\alpha F$ with 
  \begin{equation*}
\alpha \ge \left\{\begin{array}{rl}
4\quad   \text{if}\quad  a=0\\
5 \quad  \text{if}\quad  a=1\\
a\mu\quad  \text{if}\quad a\ge 2
\end{array}\right.
\end{equation*}
 \begin{equation*}
g=\left\{\begin{array}{rl}
\alpha-a-1\quad  \text{if}\quad  \mu=2\\
2\alpha-3a-2 \quad  \text{if}\quad  \mu=3\\
3\alpha-6a-3\quad  \text{if}\quad \mu=4
\end{array}\right.
\end{equation*}
and $r=\ulcorner \tau(\mu-1,g)\urcorner-1$.  In particular, for $\mu=2$,
$(X,D)$ is an $(a,g)$--Castelnuovo pair.
 \end{itemize}
\end{prop}

\subsection{Number of moduli (I)}\label{ssec:moduli}
 
Consider a pair $(X,D)$ with $D$ irreducible, smooth of genus $g>0$.
We denote by $\mathcal X_g$ the set of all these pairs. 
Given $(X,D)\in \mathcal X_g$, we have the
rational \emph{moduli map} $\mu_D: \vert D\vert \dasharrow \mathcal M_g$.
The dimension of the image of $\mu_D$ is called the \emph{number of
moduli} of  $(X,D)$, denoted by $\mu(X,D)$. 

\begin{prob}\label{prob:moduli} 
Given $g$, what is the maximum of $\mu(X,D)$ 
as $(X,D)$ varies in  $\mathcal X_g$?
\end{prob}

One might expect that, the larger the dimension of $\vert D\vert$, the larger 
$\mu(X,D)$. This is not exactly  the case as we will see by looking at 
$\mu$--Castelnuovo's pair with $2\le \mu\le 4$.

\subsubsection{Veronese pairs}\label {ssec:veronese}

Consider a $m$--Veronese pair $(X,D)$, so that $g={{m-1}\choose 2}$.
This includes case (i) of Proposition \ref {prop:list}  
with $\mu(X,D)$ maximal.
A classical theorem of M. Noether (see \cite  [\S 3] {C1}) asserts that
two smooth plane curves of degree $m$ are isomorphic if and only
if they are projectively equivalent. As a consequence we have:

\begin{prop}\label{prop:modver} If $(X,D)$ is a $m$--Veronese pair with $m\ge 3$, then
$\mu(X,D)=g+3m-9$. 
\end{prop}

The moduli map is dominant if and only if $m=4$, whereas, for $m>>0$,
$\mu(X,D)=o(g)$.

\subsubsection{Castelnuovo pairs}\label {ssec:castpairs}

Consider an $(a,g)$--Castelnuovo pair $(X,D)$, with $g\ge 2$, which is then a $2$--Castelnuovo pair. 
The curve $D$ is hyperelliptic
hence the image of $\mu_D$ is contained in the hyperelliptic locus $\mathcal H_g$ in $\mathcal M_g$
and therefore $\mu(X,D)\le 2g-1$. 

\begin{prop}\label{prop:modcast} If $(X,D)$ is an $(a,g)$--Castelnuovo pair, then 
$\mu(X,D)=2g-1$, i.e. ${\rm Im}(\mu_D)=\mathcal H_g$. 
\end{prop}
\begin{proof} We have $X=\Bbb F_a$ and $D=2E+(a+g+1)F$. Consider the exact sequence 
$0\to T_D\to T_X|_D\to N_{D,S}\to 0$. To prove the assertion it suffices to prove that
\begin{equation}\label{eq:4}
\dim(\text{Im}(H^0(D,N_{D,X})\to H^1(D, T_D)))= 2g-1.
\end{equation} 
We have  $h^0(D,N_{D,X})=r=3g+5$ and $h^ 0(D, T_D)=0$. So \eqref {eq:4} is equivalent to
$h^0(T_X|_D)=g+6$.  Consider the structure morphism $f: X\to \Bbb P^ 1$ and let $T_f$ be the relative tangent
sheaf. We have the exact sequence $0\to T_f|_D\to T_X|_D\to\mathcal O_X(2F)|_D\to 0$, from which
 $\deg (T_f|_D)=2g+2$, hence $h^1(D,T_f|_D)=0$ and therefore  $h^0(D,T_X|_D)=h^0(D,T_f|_D)+h^0(D,\mathcal O_D(2F))=g+6$,
 as needed. \end{proof}

Next we consider $\mu$--Castelnuovo pairs as in part (ii) of Proposition \ref {prop:list} 
with $3\le \mu \le 4$. 
The analysis of the moduli maps in these cases could be done, as in
the proof of Proposition \ref {prop:modcast}, by studying the coboundary map in \eqref {eq:4}. 
There is however a quicker way, which parallels  Noether's theorem for plane curves. 

\begin{prop}\label{prop:modmu} Let $(X,D)$ be a $\mu$--Castelnuovo pair as in part (ii) of Proposition \ref {prop:list} 
with $3\le \mu \le 4$ and $g\ge 4$. 
Then two smooth curves $C,C'\in \vert D\vert$ are isomorphic if and only
if there is an automorphism $\omega$ of $X\cong \Bbb F_a$ such that $C'=\omega(C)$.
Accordingly 
\begin{equation}
 \label{eq:dimmod}
\mu(X,D)=\left\{\begin{array}{rl}
2g+1\quad  \text{if}\quad  \mu=3   \quad  \text{and}\quad  a=0   \\
2g+2-a\quad  \text{if}\quad  \mu=3   \quad  \text{and}\quad  a>0   \\
\frac {5g}3+3  \quad \text{if}\quad  \mu=4   \quad  \text{and}\quad  a=0   \\
\frac {5g}3+4-a \quad \text{if}\quad  \mu=4   \quad  \text{and}\quad  a>0   \\
\end{array}\right.
\end{equation}
In particular, for $\mu=3$ and $0\le a\le 1$, the image of $\mu_D$ is the whole
trigonal locus. 
\end{prop}
\begin{proof} Consider the case $\mu=3$.  Then $D\equiv 3E+\alpha F$ with $\alpha\ge 3a$. 
Set $H\equiv D+K=E+(\alpha-a-2)F$. Then  $\phi_{\vert H\vert}$ is a morphism
mapping $X$ to a rational normal scroll $S$  in
$\Bbb P^ {g-1}$, and the smooth curves
in $\vert D\vert$ are mapped to canonical curves. Two of such curves $C, C'$ are isomorphic if and
only if there is a projective transformation $\omega$ of $\Bbb P^ {g-1}$ such that
$C'=\omega(C)$. Since $\omega(S)=S$, the first assertion follows. 

Note that $r=\tau(2,g)-1=2g-7$. The automorphisms group of $\Bbb F_a$ has dimension
$a+5$ if $a>0$ and $6$ if $a=0$, which explains the first two lines of \eqref {eq:dimmod}.

Look now at the case $\mu=4$.  Assume first $a\ge 3$.
Then $D\equiv 4E+\alpha F$ with $\alpha\ge 4a$. 
Set $H\equiv E+\beta F$, with $\beta$ verifying 
 \begin{equation}
 \label{eq:beta}
 \alpha\le 4\beta\le 2\alpha -2a-2.
\end{equation}
Since $\alpha\ge 4a$ and $a\ge 3$, certainly such a $\beta$ exists.
In addition $\beta\ge \alpha/4\ge a$, hence  $\phi_{\vert H\vert}$  maps $X$ to a rational normal
scroll  in $\Bbb P^ s$, with $s=2\beta-a+1$.
The curves in $\vert D\vert$ map to curves of degree $n= 4\beta+\alpha-4a$. 
Note that $n-1=3(s-1)+\epsilon$, with $\epsilon=\alpha-a-2\beta-1$.
By \eqref {eq:beta}, one has $0\le \epsilon<s-1$. Then
the maximal genus of 
curves of degree $n$ in $\Bbb P^ s$ is 
$3\alpha-6n-3=g$ (see, e.g.,  \cite [p. 527] {GH}).
Hence the smooth curves in $\vert D\vert$ are Castelnuovo curves in $\Bbb P^ s$.
By a result of Accola (see \cite {Ac} and also \cite [Teorema (2.11)]{C1}), 
two smooth curves $C,C'\in \vert D\vert$ are isomorphic if and only if they
are projectively equivalent in $\Bbb P^ s$. The conclusion is as for $\mu=3$. 

In case $0\le a\le 2$, the same argument as above applies if $\alpha\ge 5+2a$, 
since in this case still there is an integer $\beta$ verifying \eqref {eq:beta}.
So we are left to consider the cases $0\le a\le 2$ with $\alpha\le 4+2a$.
Then $\alpha= 4+2a$ by Proposition
\ref {prop:list}, (ii), for $a=0,2$.
In case $a=1$ also $\alpha= 4+2a=6$, because $\mu=4, \alpha=5$
does not correspond to a $4$--Castelnuovo pair.
The argument is similar
to the above and therefore we will be brief. For $a=0,2$
we map $\Bbb F_a$ to a quadric $S$ in $\Bbb P^ 3$. Then the curves in $\vert D\vert$ are complete intersections of $S$ with a surface
of degree $4$. Again two smooth curves $C,C'\in \vert D\vert$ are isomorphic if and only if they
are projectively equivalent in $\Bbb P^ 3$ (see \cite{CL} or \cite [Corollario (4.8)]{C1}).
If $a=1$ 
then $\Bbb F_1$ birationally maps to the plane by contracting $E$ and $\vert D\vert $ is the proper transform of the linear
system of curves of degree 6 with a double base point. Two such curves with only one node
are birational if and only if they are projectively equivalent in $\Bbb P^ 2$ (see \cite [Osservazione (2.19)]{C1})
and the conclusion is as above.  \end{proof}

The last assertion in Proposition \ref {prop:modmu} is no news: indeed it goes back to Maroni \cite {Ma}.

\subsection{Number of moduli (II)}\label{ssec:moduli2}
In this section we  answer  Problem \ref {prob:moduli}.

\begin{thm}\label{thm:moduli} Let $(X,D)$ be a minimal pair with $D$ a 
smooth curve of genus $g\ge 22$. Then $\mu(X,D)\le 2g+1$
and equality holds if and only if $(X,D)$ is a $3$--Castelnuovo pair with 
$X\cong \Bbb F_a$ and $0\le a\le 1$,
in which case the image of $\vert D\vert$ via $\mu_D$ is the trigonal
locus in $\mathcal M_g$.
\end{thm}

\begin{proof} By Remark \ref {rem:g}, we may assume $X$ has negative Kodaira dimension,
otherwise $r\le g$. 
If $q>0$, there is a curve $C$ of genus $q$ and a surjective morphism $f: X\to C$,
hence all curves in $\vert D\vert$ map to $C$ and therefore $\mu(X,D)\le 2g-2$
(see \cite {CGT}). So we may assume $q=0$.

If $\vert D\vert$ is composed with an involution, then 
the general curve $D\in \vert D\vert$ has a non--constant morphism $D\to C$
to a curve. If $C$ is rational, then $D$ is hyperelliptic and 
$\mu(X,D)\le 2g-1$, if $C$ is irrational, one has $\mu(X,D)\le 2g-2$ as above. 
Thus, if $\mu(X,D)\ge 2g$ we may assume that $(X,D)$ verifies
the hypotheses of Castelnuovo's Theorem \ref {thm:Castthm}.

By Castelnuovo--Enriques Theorem \ref {thm:CastEnr}, we may assume that $r\le 3g+5=\tau(1,g)-1$.
If $r \ge 2g+8=\tau(2,g)$, then $(X,D)$ is a $2$--Castelnuovo pair. By Propositions \ref {prop:list},
\ref {prop:modver}  and \ref {prop:modcast},  the image of $\mu_D$ is $\mathcal H_g$ and $\mu(X,D)=2g-1$.
If $\tau(2,g)-1\ge r\ge \tau(3,g)= \frac {5g+1}3 +9$, then $(X,D)$ is a $3$--Castelnuovo pair.
By Propositions \ref {prop:list}, \ref {prop:modver}  and \ref {prop:modmu}, 
the maximum of $\mu(X,D)$ is attained if $(X,D)$  is as in (ii) of Proposition \ref {prop:list} with $0\le a\le 1$.
In this case $\mu(X,D)=2g+1$ and  the image of $\vert D\vert$ via $\mu_D$ is the trigonal
locus. If $\tau(3,g)>g\ge \tau(4,g)=\frac {3g+1}2+11$, then
the maximum of $\mu(X,D)$ is $\frac {5g}3+3$ (see Proposition \ref {prop:modmu})
which is smaller than $2g+1$ if $g\ge 7$. Finally, if $r<\frac {3g+1}2+11$, then also
$\mu(X,D)\le \frac {3g+1}2+11$ and this is smaller than $2g+1$ if $g\ge 22$. \end{proof}

\end{document}